%% file: WildRamificationArxiv.tex
\documentclass[12pt]{amsart}
\input{preamble.tex}

\title[Analytic Langlands Correspondence for $\PGL_2$ and $\Pone$ with Wild Ramification]{On the Analytic Langlands
Corrrespondence for $\PGL_2$ in Genus 0 with Wild Ramification}
\author{Daniil Klyuev, Atticus Wang}

\begin{document}
\maketitle
\begin{abstract}
    The analytic Langlands correspondence was developed by Etingof, Frenkel and Kazhdan in \cite{EFK1,EFK2,EFK3,EFK4}. For a curve $X$ and a group $G$ over a local field $F$, in the tamely ramified setting one considers the variety $\Bun_G$ of stable $G$-bundles on $X$ with Borel reduction at a finite subset $S\subset X$ of points. On one side of this conjectural correspondence there are Hecke operators on $L^2(\Bun_G)$, the Hilbert space of square-integrable half-densities on $\Bun_G$; on the other side there are certain opers with regular singularities at $S$. In this paper we prove the main conjectures of analytic Langlands correspondence in the case $G = \PGL_2$, $X=\PP^1_{\CC}$ with wild ramification, i.e.~when several points in $S$ are collided together.
    
\end{abstract}
\tableofcontents
\bibliographystyle{hplain}

\section{Introduction}

\subsection{Analytic Langlands correspondence}

In \cite{EFK1,EFK2,EFK3,EFK4}, an analytic version of the Langlands correspondence was formulated for curves over local fields, motivated in part by the works~\cite{BK}, \cite{Ko}, \cite{La}, \cite{Te}. The general setup for the tamely ramified case, which we recount for completeness, is as follows. Let $X$ be a smooth projective irreducible curve over a local field $F$, let $G$ be a connected reductive algebraic group over $F$, and let $S$ be a finite set of $F$-points in $X$. By $\Bun_G(X,S)$ we denote the algebraic stack of $G$-bundles $\cE$ on $X$ with Borel reduction on $S$. On the automorphic side, one considers the Hilbert space $\cH$ of square-integrable half-densities on  the open dense substack of stable bundles in $\Bun_G(X,S)$; in \cite{EFK2}, a commutative algebra of Hecke operators was constructed, initially only on a dense subspace of $\cH$, but conjectured to extend by continuity to compact normal operators on $\cH$. On the spectral side, for the case $F=\CC$, it is conjectured that the joint spectrum of Hecke operators should correspond to the set of $G^{\vee}$-opers with real monodromy, where $G^{\vee}$ is the Langlands dual group of $G$.

In \cite{EFK3}, this recipe was implemented for $G = \PGL_2$, $X = \PP^1$, and $S$ a set of distinct $F$-points $t_0,\dots,t_{m+1}$ in $X$, where $m\ge 1$ (a necessary condition for existence of stable bundles). In this case, a $G$-bundle with Borel reduction at $S=\{t_0,\ldots,t_{m+1}\}$ is simply a rank 2 vector bundle, up to tensoring by line bundles, with distinguished 1-dimensional subspaces in the fibers above the marked points $t_0,\dots,t_{m+1}$. Such bundles are called \emph{quasiparabolic bundles}. In this case, the moduli stack of stable quasiparabolic bundles is known to be a smooth, quasiprojective variety, and is the union of two connected components, bundles of degree 0 and 1, respectively. There are isomorphisms identifying the two components, given by Hecke modification at any of the marked points; so it suffices to consider the degree 0 component $\Bun_G^0$. This space could be parametrized birationally by $\PP^{m-1}$ (\cite{EFK3}, Lemma 3.1): by fixing the lines above $t_0 = 0$ and $t_{m+1}=\infty$, a generic quasiparabolic bundle is uniquely given by $m$ elements of $F$, each specifying the line above $t_1,\dots,t_m$, up to simultaneous scaling. Therefore, $\cH = L^2(\Bun_G^0) = L^2(\PP^{m-1})$ is the space of square-integrable half-densities on $\PP^{m-1}$ (sections of $|\cK|$, where $\cK = \cO(-m)$ is the canonical bundle and $\abs{\cdot}$ is a norm map equal to the usual absolute value in the case of $F=\C$). An element $\psi\in \cH$ can therefore be realized as a complex-valued function $\psi(y_1,\dots,y_m)$ on $F^m\backslash\{0\}$, such that $\psi(z \mathbf{y})=|z|^{-m}\psi(\mathbf{y})$ for any $z\in F^\times$. 

Under this parametrization, the Hecke operators take the following explicit form. For each $x\in \PP^1\backslash\{t_0,\ldots,t_{m+1}\}$, the Hecke operator $H_x$ is given by
\begin{equation}
\label{hecke-original}
	(H_x\psi)(y_1,\dots,y_m) =  \paren{\prod_{i=0}^m |t_i-x|}\cdot\int_{\CC} \psi\paren{\frac{t_1s-xy_1}{s-y_1},\cdots,\frac{t_ms-xy_m}{s-y_m}}\frac{|s|^{m-2} \abs{ds}^2}{\prod_{i=1}^m |s-y_i|^2}.
\end{equation}
It was shown in \cite{EFK3}, Section 3 that $H_x$ extend to compact, self-adjoint, mutually commuting operators on $\cH$, with zero common kernel. Importantly, this relies on the fact that $H_x$ is given by integrating certain unitary operators $U_{s,x}$ over $s\in F$.

Now let $F=\C$. The next important step is the differential equation for Hecke operators:
\begin{equation}
\label{EqDifferentialEquationClassic}
\paren{\partial_x^2+\sum_{i\geq 0}\frac{1}{4(x-t_i)^2}}H_x-H_x\sum_{i\geq 0}\frac{G_i}{x-t_i}=0.
\end{equation}
Here $G_i$ ($0\le i\le m$) are certain commuting holomorphic differential operators on $\cH$. One of the consequences of~\cref{EqDifferentialEquationClassic} is that although these $G_i$ are unbounded operators, they commute with $H_x$ in a certain well-defined sense, so that we get a good spectral problem for both Hecke and differential operators (since Hecke operators are compact self-adjoint). A second consequence is that in this case ($\PGL_2$ and $\PP^1$ over $\CC$), the joint eigenvalues $\beta_k(x)$ (real-valued and continuous in $x$, labeled by $k\in\NN$) satisfy the following differential equation (\cite{EFK3}, Corollary 4.14):
 \begin{equation}
 \label{diff-eq-original}
 	\paren{\partial_x^2 + \frac14\sum_{i=0}^m \frac{1}{(x-t_i)^2} - \sum_{i=0}^m \frac{\mu_{i,k}}{x-t_i}}\beta_k(x) = 0,
 \end{equation}
which is an $\SL_2(\CC)$-oper (i.e.~no $\partial_x$ term); $\SL_2$ is Langlands dual to $\PGL_2$. Here $\mu_{i,k}\in \CC$ are eigenvalues of $G_i$ on the eigenfunction $\psi_k$ corresponding to $\beta_k$ (in particular it was shown that the joint spectrum of $H_x$ is simple). Moreover, the monodromy representation of such a differential equation (where $\mu_{i,k}$ are now parameters in $\CC$) lands in $\SL_2(\RR)$ up to conjugation if (and, partially, only if) they come from a joint eigenfunction of Hecke operators (\cite{EFK3}, Theorem 4.15), thus establishing analytic Langlands correspondence.

\subsection{Summary of our paper}

In this paper we investigate what happens when we collide several points among $t_i$, i.e. when $S$ is no longer a reduced divisor. For example, suppose we merge only $t_0$ and $t_1$. One obvious way of obtaining a limit of Hecke operators is to simply set $t_0=t_1$ in \cref{hecke-original}; this corresponds to choosing two lines in the fiber of the quasiparabolic bundle above the closed point $t_0=t_1$. However, the resulting Hecke operators will not be compact and will have continuous spectrum, hence they will have no eigenvectors.

Instead, we will make $t_0=t_1$ a non-reduced point, in this case a $\CC[\eps]/(\eps^2)$-point. A generic line in its fiber is given by $(1,u_0+u_1\eps)$, so that in \cref{hecke-original} one should change variables $y_0,y_1$ by $u_0=y_0$, $u_1 = \frac{y_1-y_0}{t_1-t_0}$. In order to have a well-defined limit as $t_1\to t_0$, we should also use a twisted version of Hecke operators, whose twisting parameters are sent to infinity in an appropriate way. 

We carry this out in \cref{SecLimitAndDefinition}, obtaining limits of Hecke operators $H_x$. We use this computation as a motivation for the following definition of modified Hecke operator:
    \begin{align*}
  \HH_x\psi(\bu_0,\ldots,\bu_m)=\int_{\C}\psi\bigg(\frac{t_0+\eps_0-x}{s-\bu_0},\ldots,\frac{t_m+\eps_m-x}{s-\bu_m}\bigg)\frac{\exp\bigg(\sum\chi_i\big(\log(s-\bu_i)\big)\bigg)ds\ovl{ds}}{\prod_{i=0}^m \abs{s-u_i^{(0)}}^{2n_i+2}}
    \end{align*}

Here $\bu_0,\ldots,\bu_m$ parametrize fibers over non-reduced points corresponding to $t_0,\ldots,t_m$, $\chi_i$ are certain $\RR$-linear maps to $i\RR$. Non-modified Hecke operator $H_x$ is obtained as $\HH_x$ times a certain function of $x$. 

In \cref{SecUnitaryRep} we prove that $H_x$ given by integrating some unitary representation $U_{s,x}$ over $s\in \CC$ (\cref{unitary}). The measure is the same as in~\cite{EFK3,EFK4}. In \cref{SecBounded} we prove that $H_x$ are bounded self-adjoint operators on $\cH$ that commute with each other. Moreover, in \cref{SecCompactness} we show that $H_x$ are compact and norm-continuous. In \cref{SecSpectralDecomp} we show that $\{H_x\}$ have zero common kernel, and therefore they have a joint discrete spectrum with finite-dimensional eigenspaces. In other words, we recover the main properties of Hecke operators required for establishing analytic Langlands correspondence in our case.

Wildly ramified case was briefly considered in~\cite{EFK4}, Section 2.14. It can be shown that our approach fits into a general definition of ramified analytic Langlands correspondence. Moreover, we prove that there is strong limit $\lim_{\eps\to 0}U_{\eps}^{-1}H^{\lambda(\eps)}_xU_{\eps}=H_{\text{ramified},x}$, where $U_{\eps}$ is the unitary operator corresponding to the coordinate change $(y_0,\ldots,y_n)\mapsto (u_0,\ldots,u_n)$ below and $H_{\lambda(\eps)}$ are twisted Hecke operators with twisting parameters and marked points depending on $\eps$.  

In \cref{SecGaudin} we define a family of commuting differential operators $G_i^{(j)}$ and prove the analogue of the differential equation~\eqref{EqDifferentialEquationClassic}. As a corollary, we obtain a differential equation on eigenvalues $\beta_k(x)$ that corresponds to an $SL_2$-oper, partially establishing analytic Langlands correspondence in this case. We also prove that an eigenvector $\psi\in \mc{H}_k$ satisfies \[G_i^{(j)}\psi=\mu_{i,j,k}\psi\] in the sense of distributions. We expect that the $D$-module corresponding to this system of equations on $\psi$ is irreducible, hence $\psi$ is unique up to scaling and $\mc{H}_k$ is one-dimensional. We also expect that $\psi$ are smooth on an open subset of $\Bun_G$ and that $G_i^{(j)}$ have a natural self-adjoint extensions that strongly commute with each other and with Hecke operators.

However, the important difference with~\cite{EFK3} is that the differential equation on $\beta_k(x)$ will no longer have regular singularities at $t_i$. Namely it will have irregular singularities at the merged points (wild ramification). So the condition of real monodromy is not enough, and there should be a condition on the Stokes data or asymptotic expansion of solutions at irregular singularities. This is currently under investigation.

It can be computed that the limit of generating function for twisted Gaudin operators after coordinate change, $\sum_{i=0}^m \frac{U_{\eps}^{-1}G_{i,\lambda(\eps)}U_{\eps}}{x-t_i}$, gives the generating function $\sum_{i=0}^m\sum_{j=0}^{n_i} \frac{G_i^{(j)}}{(x-t_i)^{j+1}}$ defined in \cref{SecGaudin}. Here $n_i+1$ is the multiplicity of point $t_i$.


\section{Preliminaries}

Throughout the paper we will assume that $F=\C$. Here we collect background material and auxilliary lemmas that are used in the proof of the main results.

\subsection{Non-reduced point with parabolic structure}

Fix an integer $n\ge 0$ and write $\CC[\eps] = \CC[x]/(x^{n+1})$. As mentioned in the introduction, let us consider a $\CC[\eps]$-point $t_0$ on $\PP^1$ with parabolic structure, i.e.~there is a chosen rank-1 free $\CC[\eps]$-submodule of $\CC[\eps]^{\oplus 2}$, the fiber of the quasiparabolic bundle $\cO^{\oplus 2}$ above $t_0$. Generically, one may assume it is the line spanned by $(1, \sum_{k=0}^{n} u_{k} \eps^k)$. 

Let $x\neq t_0$ be a closed point, and $s$ a line in the fiber above $x$. After Hecke modification at $(x, s)$ (and rewriting in terms of the original parametrization, see \cite{EFK3}, sections 3.1, 3.2), the line $(1, \sum_{k=0}^{n} u_{k} \eps^k)$ becomes $\paren{\sum_{k=0}^{n} u_{k}\eps^k - s, t_0-x+\eps}$. Part (a) of \cref{identities} below shows that this is just the line $(1, -\sum_{k=0}^{n}  \frac{1}{k!} \partial^k(\frac{t_0-x}{s-u_0})\eps^k)$.

\begin{defn}
	Consider the field $\CC(s,x,t_0,u_{i,j})$ generated formally by these symbols, where $0\le j\le i\le n$. Write $u_i := u_{i,i}$. Define a derivation $\partial$ on this field, defined by $\partial s = \partial x = 0$, $\partial t_0 = 1$, and $\partial u_{i,j} = (j+1)u_{i+1, j+1}$. Finally, for purely imaginary numbers $a_j$, define $D=-2i\sum_{j=1}^n \frac{a_j}{j!} \partial^j$, where $i$ is an imaginary unit.
\end{defn}

\begin{prop}
\label{identities}
We have the following identities:
\begin{enumerate}[(a)]
	\item $t_0-x+\eps = \paren{\sum_{k=0}^{n} \frac{1}{k!}\partial^k(\frac{t_0-x}{s-u_0}) \eps^k}\paren{s-\sum_{k=0}^{n} u_{k}\eps^k}$; \\
	\item For $0\le m\le n$, $\frac{1}{m!}\partial^m(\log(s-u_0)) = [\eps^m]\log(s-\sum_{k=0}^{n} u_{k}\eps^k)$.
\end{enumerate}
\end{prop}

\begin{proof}
For any $X\in \CC(s,x,t_0,u_{i,j})$, consider its Taylor series $T(X) = \sum_{k=0}^{n} \frac{1}{k!}\partial^k (X)\eps^k$. It is easily checked that $T(X_1X_2) = T(X_1)T(X_2)$ and $T(\log(C-X))=\log(T(C-X))$, for any constant $C$ (i.e.~$\partial C = 0$). Part (a) is simply 
$T(t_0-x)= T(\frac{t_0-x}{s-u_0})T(s-u_0)$, and part (b) follows from $T(\log(s-u_0))=\log(T(s-u_0))$.
\end{proof}

\begin{rem}
    In particular, $D\log(s-\sum_{k=0}^{n} u_{k}\eps^k)$ can be computed as the integral
    \[-\frac{1}{2\pi}\int_{\abs{z}=1}(\sum a_j z^{-j-1})\log(s-\sum_{k=0}^n u_kz^k)dz.\]
\end{rem}

It follows that $i\Re D(\log(s-u_0))=\chi(\log(s-\sum_{k=0}^n u_k\eps^k))$, where $\chi$ is an $\mathbb{R}$-linear map $\chi:\C[\eps]\to i\RR$ given by $\chi(\sum b_j\eps^j)=\sum a_j \Re b_j$. 

Below we will consider $\chi$ of the form $\chi(\sum_{j=0}^n b_j\eps^j)=i\Re\sum_{i=0}^n c_jb_j$, where $c_j$ are not required to be real. However, we require $\Im c_0$ to be integer, so that $\exp\circ\chi\circ\log$ is well-defined. We also require $c_n$ to be nonzero below. Nonzero $\Re c_0$ corresponds to twisting, nonzero $\Im c_0$ corresponds to taking non-spherical principal series representations of $\PGL_2(\C)$.
\begin{rem}
The $\chi$ of the form $\sum a_j\Re b_j$ arise from taking the limit of twisted Hecke operators in a specific way. We expect that by changing the limiting procedure for $t_i$ or $\lambda_i$, we can obtain any imaginary functional $\chi$ on $\C[\eps]$ with $c_0$ real. If we take limit of Hecke operators corresponding to any principal series representation of $\PGL_2(\C)$, we expect to get any imaginary $\chi$ as above.
\end{rem}

\subsection{Representations of $\PGL_2(\C[\eps])$.}
\label{SecRepresentationsOfGEps}
Recall that the group $\PGL_2(\CC)$ has a natural right action on square-integrable half-densities on $\PP^1$, given by
\[\begin{pmatrix}
a & b \\
c & d	
\end{pmatrix} f(z) = \frac{|ad-bc|}{|cz+d|^2} f\paren{\frac{az+b}{cz+d}}.\]

Let $\CC[\eps] = \CC[x]/(x^{n+1})$. The group $\PGL_2(\CC[\eps])$ acts naturally on $\PP^1(\CC[\eps])$ by
\[\begin{pmatrix}
    a & b \\
    c & d
\end{pmatrix}(z) = \frac{az + b}{cz+d}.\]
Suppose we identify $z = u_0+\dots+u_n\eps^{n}$ with $(u_0,\dots,u_n)\in \CC^{n+1}$, with the usual measure. Then for any $\chi:\C[\eps]\to i\RR$ as above we can define a right unitary representation $\rho$ of $\PGL_2(\CC[\eps])$ on $L^2(\CC^{n+1})$, by
\[\rho(g)f(z) = f(gz)\cdot \abs{\frac{\det g_0}{(c_0u_0+d_0)^2}}^{n+1} \cdot\exp\bigg(\chi\big(\log(cz+d)-\tfrac12\log(\det g)\big)\bigg),\]
where $g = \begin{pmatrix}
    a & b \\
    c & d
\end{pmatrix}\in \PGL_2(\CC[\eps])$, and $g_0 = \begin{pmatrix}
    a_0 & b_0 \\
    c_0 & d_0
\end{pmatrix}$ is the constant part of $g$. Since we assumed that $\chi(\sum b_j\eps^j)=i\Re\sum c_jb_j$ with nonzero $c_n$, it is well-known that this representation is irreducible.

We also note the following. Let $B$ be the Borel subgroup of $\PGL_2$ of upper-triangular matrices, so that $\Pone(\C[\eps])\cong \PGL_2(\C[\eps])/B(\C[\eps])$. Consider the line bundle $\mc{L}=O(-1)\otimes \mc{L}_{\eps}$ on $\Pone(\C[\eps])$, where $\mc{L}_{\eps}$ corresponds to the one-dimensional representation of $B(\C[\eps])$ given by $\begin{pmatrix}
    a & b\\
    0 & a^{-1}
\end{pmatrix}\mapsto \exp(\chi(\log(a)))$. Then $\rho$ corresponds to the action of $\PGL_2(\C[\eps])$ on the square-integrable sections of line bundle $\mc{L}\otimes \ovl{\mc{M}}$, where $\mc{M}=O(-1)\otimes \mc{L}_{\eps}^{-1}$.

\subsection{Lemmas about Lie groups}

\begin{lem}
\label{pavel}

Let $G_1,G_2$ be Lie groups, and $K$ a closed subgroup of $G_1\times G_2$  that surjects onto both $G_1$ and $G_2$. Let $L\subset G_2$ be the kernel of $K\into G_1\times G_2\to G_1$. Then $L\lhd G_2$, and $K$ is the preimage in $G_1\times G_2$ of  the graph of a smooth homomorphism $f:G_1\to G_2/L$.
\end{lem}

\begin{proof}
	In the case $L = 1$, $K\to G_1$ is an isomorphism, so $f$ is given by its inverse composed with the map $K\into G_1\times G_2\to G_2$.  In general, suppose $(1,\ell)\in L \subset K$. For any $g_2\in G_2$, there exists $(g_1,g_2)\in K$, so $(g_1,g_2)^{-1}(1,\ell)(g_1,g_2) = (1, g_2^{-1}\ell g_2)\in K$, so $L$ is normal. Let $K'$ be the image of $K$ in $G_1\times (G_2/L)$. Then we may apply the $L=1$ case to $K'$, and $K$ is the preimage in $G_1\times G_2$ of the graph of a $f:G_1\to G_2/L$.
\end{proof}

\begin{lem}
\label{C[eps]-maps}
	The nontrivial Lie group homomorphisms $f:G(\CC[\eps])\to G(\CC)$ are all of the form $\psi\circ \pi$, where $\pi:G(\CC[\eps])\to G(\CC)$ is projection to constant term, and $\psi\in \Aut(G(\CC))$. In particular, they are surjective.
\end{lem}

\begin{proof}
	Pass to Lie algebra homomorphism $df: \sl_2(\CC[\eps])\to \sl_2(\CC)$. The restriction of $df$ on $\sl_2(\CC)\subset \sl_2(\CC[\eps])$ is an inner automorphism, since $\sl_2(\CC)$ is simple and $df$ is not identically zero. Since every element in $\sl_2(\CC[\eps])$ with zero constant term is ad-nilpotent, we conclude that they lie in the kernel of $df$. So $f$ is an automorphism precomposed with projection as well.
\end{proof}

\section{Limits of Hecke operators}

\label{hecke}

\subsection{Twisted Hecke operators}

Let $t_0,\dots,t_{m+1}\in \PP^1$. Without loss of generality, let us fix $t_{m+1}=\infty$. Let $x\in \PP^1$, $x\neq t_i,\infty$. Let $\lambda=(\lambda_0,\dots,\lambda_{m+1})$ be twisting parameters, which are complex numbers satisfying $\Re\lambda_j=-1$. 

For any purely imaginary number $c$, we can consider the Hilbert space $\cH = L^2(\PP^{m-1}_{\CC}, |\cK|^{1+c})$, whose elements we view as complex-valued functions $\psi(y_1,\dots,y_m)$ on $\CC^m\backslash\{0\}$, homogeneous of degree $-m(1+c)$. They may also be viewed as functions $\psi(y_0,y_1,\dots,y_m)$ which are both translation-invariant and homogeneous of degree $-m(1+c)$, where geometrically $y_i$ parametrize the quasiparabolic lines above $t_i$; this interpretation has more symmetry and makes formulas nicer.

The \emph{twisted Hecke operators} $H_x^{\lambda}$ introduced in~\cite{EFK4} are given by
\begin{equation}
\label{twisted-hecke}
	(H_x^{\lambda}\psi)(y_0,\dots,y_m) =  \paren{\prod_{i=0}^m |t_i-x|^{-\lambda_i}}\cdot\int_{\CC} \psi\paren{\frac{t_0-x}{s-y_0},\cdots,\frac{t_m-x}{s-y_m}}\prod_{i=0}^m |s-y_i|^{2\lambda_i}dsd\bar{s}.
\end{equation}
It is easy to check that $H_x^{\lambda}$ is a linear map which maps functions homogeneous of degree $-\lambda_{m+1}+\sum_{i=0}^m \lambda_i$ to functions homogeneous of degree $2+\sum_{j=0}^{m+1}\lambda_j$. For simplicity, we will limit ourselves to the case when $\lambda_{m+1} = 0$ and $\sum_{i=0}^m \lambda_{i} = 0$, so that functions in the domain and codomain of $H_x^{\lambda}$ have the same homogeneity degree $-m$. 

If we omit the constant term $\prod |t_i-x|^{-\lambda_i}$ in \cref{twisted-hecke}, the formula gives so-called \emph{modified} Hecke operators, denoted by $\HH_x^{\lambda}$.

\subsection{Taking the limit}
\label{SecLimitAndDefinition}
Suppose we wish to merge points $t_0,\dots,t_n$, where $n\le m$. For simplicity, we let the other points remain distinct, but one can merge more than one group of points by the same procedure. Let $a_1,\ldots,a_n$ be a sequence of imaginary numbers such that $a_n$ is nonzero. For $0\le i\le n$, $1\le j\le n$ define
\begin{equation}
\label{twisting-param}
\lambda_i^{(j)} = \frac{a_j}{\prod\limits_{\substack{0\le k\le j \\ k\neq i}} (t_i-t_k)},
\end{equation}
when $i\leq j$ and zero else. 

Let $\lambda_i=-1+\sum_j \lambda_i^{(j)}$ when $i\leq n$ and $-1$ when $i>n$. In the limiting process, we will make $t_i-t_{i-1}$ ($1\le i\le n$) all equal, real numbers $\delta$, as we take the limit $\delta\to 0$.



Make a change of variables
\begin{equation}
\label{reparam}
u_i = \sum_{0\le j\le i} \frac{y_j}{\prod\limits_{\substack{0\le k \le i \\ k \neq j}} (t_j-t_k)}.
\end{equation}
In fact, define variables $u_{i,j}$, $0\le j\le i\le n$ recursively, as follows: $u_{i,0} = y_i$, $u_{i,j} = \frac{u_{i, j-1} - u_{i-1,j-1}}{t_i - t_{i-j}}$. Then it is easy to see $u_i=u_{i,i}$. We also let $u_i = y_i$ for $n+1\le i\le m$ for simplicity. Note that now an element $\psi = \psi(u_0,\dots,u_m)\in \cH$ will still be homogeneous of degree $-m$, but translation invariant only in the variables $u_0,u_{n+1},\dots,u_m$ while $u_1,\dots,u_n$ remain fixed.

\begin{prop}
\label{formula}
The strong limit $\HH_x$ of the modified Hecke operator after coordinate change $U_{\eps}\HH_x^{\lambda}U_{\eps}^{-1}$, as $\delta\to 0$, is given by
\begin{align*}
&(\HH_x \psi)(u_0,\dots,u_m) \\
=&\ \int_{\CC} \psi\paren{\frac{t_0-x}{s-u_0}, \partial\paren{\frac{t_0-x}{s-u_0}}, \dots, \frac{1}{n!} \partial^{n}\paren{\frac{t_0-x}{s-u_0}}, \frac{t_{n+1}-x}{s-u_{n+1}},\dots}\frac{\exp(i\Re D\log(s-u_0))dsd\bar{s}}{|s-u_0|^{2n+2}\prod_{k=n+1}^m |s-u_k|^2}.
\end{align*}
Here $U_{\eps}$ is the unitary operator corresponding to the coordinate change $(y_0,\ldots,y_n)\mapsto (u_0,\ldots,u_n)$.
\end{prop}

\begin{rem}
We can collide several clusters of points, doing the same procedure as above for each collided point, and the proof is the same. If we leave three points unglued, then we still have a similar strong limit.
\end{rem}

\begin{proof}
Modified Hecke operators are uniformly bounded, the operator $\HH_x$ is also bounded (this does not create circular reasoning since we do not use the strong limit in this paper.) Hence it is enough to prove strong convergence on a dense subset of $\cH$. Let $\psi$ be a continuous function with compact support modulo translations and dilations. We will show that for $s\in \C$ the limit of 
\[f_{\lambda}(s)=\psi\paren{\frac{t_0-x}{s-y_0},\cdots,\frac{t_m-x}{s-y_m}}\prod_{i=0}^m |s-y_i|^{2\lambda_i}\] is
\[f(s)=\psi\paren{\frac{t_0-x}{s-u_0}, \partial\paren{\frac{t_0-x}{s-u_0}}, \dots, \frac{1}{n!} \partial^{n}\paren{\frac{t_0-x}{s-u_0}}, \frac{t_{n+1}-x}{s-u_{n+1}},\dots}\frac{\exp(i\Re D\log(s-u_0))dsd\bar{s}}{|s-u_0|^{2n+2}\prod_{k=n+1}^m |s-u_k|^2}\] and similarly for several collided points.

Using the results of~\cite{EFK4} and \cref{SecUnitaryRep} we see that the sequence of functions $f_{\lambda}(s)$ and $f(s)$ satisfies the conditions of dominated convergence theorem with $g(s)=\frac{M}{\abs{s(s-1)(s-x)}}$, where $M$ is the maximum value of $\abs{\psi}$ on the hyperplane $u_{m}=0$, $u_{m-1}=1$. It follows that $\HH^{\lambda}\psi$ tends to $\HH\psi$ as $\delta$ tends to zero.

Let $j>0$. Let us show that the limit of the term $|s-y_0|^{2\lambda_0^{(j)}}\cdots |s-y_n|^{2\lambda_j^{(j)}}$ is \[\exp(\frac{2a_j}{j!} \Re\partial^{j}\log(s-u_0)).\]
Use induction on $j$. The base case $j=0$ is clear. To reduce clutter, write $\lambda_i^{(j)}=\lambda_i$ and $j=n$ below.
In general, we have for $0<i<n$, 
\[\lambda_i = \frac{a}{\prod_{0\le k\neq i\le n} (t_i - t_k)} = \frac{1}{t_n-t_0}\paren{\frac{a}{\prod_{0<k\neq i\le n} (t_i-t_k)} - \frac{a}{\prod_{0\le k\neq i< n} (t_i-t_k)}},\]
so
\begin{equation}
\label{ratio}
\prod_{i=0}^n \abs{s-y_i}^{2\lambda_i} = \paren{\frac{\prod_{i=1}^n |s-y_i|^{2\lambda_{i, [1, n]}}}{\prod_{i=0}^{n-1} |s-y_i|^{2\lambda_{i, [0,n-1]}}}}^{\frac{1}{t_n-t_0}},
\end{equation}
where $\lambda_{i,[0,n-1]} = \frac{a}{\prod_{0\le k\neq i \le n-1} (t_i-t_k)}$ and $\lambda_{i,[1, n]} = \frac{a}{\prod_{1\le k\neq i \le n} (t_i-t_k)}$. By the induction hypothesis, the limit of the RHS of \cref{ratio} as $\delta\to 0$ is
\[\lim_{\delta\to 0} \exp\paren{\frac{2a}{(n-1)!}\Re\partial^{n-1}\frac{1}{n\delta}(\log(s-u_{1,0})-\log(s-u_{0,0}))} = \exp\paren{\frac{2a}{n!}\Re\partial^n\log(s-u_0)},\]
by using $u_{1,0} = u_{0,0} + \delta u_{1,1}$.

Let us also consider the terms $\frac{t_i-x}{s-u_i}$. Use induction on $n$ again. The base case $n=0$ is clear. The induction step is given by
\[\lim_{\delta\to 0}\frac{1}{(n-1)!} \partial^{n-1}\frac{1}{n\delta}\paren{\frac{t_1-x}{s-u_{1,0}}- \frac{t_0-x}{s-u_{0,0}}} = \frac{1}{n!}\partial^n\paren{\frac{t_0-x}{s-u_0}},\]
where we used $t_1=t_0+\delta$ and $u_{1,0} = u_{0,0} + \delta u_{1,1}$.

Note that this proof works for the several clusters of points: all coordinate changes and limits are done independently.
\end{proof}

\begin{exm}
    Let $n=2$. Then we have
\begin{align*}
\HH_x\psi(u_0,u_1,\ldots,u_m)=&\int_{\C}\psi\big(\frac{t_0-x}{s-u_0},\frac{1}{s-u_0}+u_1\frac{t_0-x}{(s-u_0)^2},\frac{t_2-x}{s-u_2},\cdots\big)\\
&\quad\quad \cdot \frac{\exp(2a_1\Re\frac{u_1}{s-u_0})ds\ovl{ds}}{\abs{s-u_0}^4\abs{s-u_2}^2\cdots \abs{s-u_m}^2}.
\end{align*}
\end{exm}
Motivated by \cref{formula} and \cref{identities} we define Hecke operators for several (possibly) glued points as follows. First, we change the notation: all points $t_0,\ldots,t_m$ can be glued. To each point $t_i$ with multiplicity $n_i+1$ corresponds its own set of variables $u_i^{(j)}$, where $i=0,\ldots,m$ and $0\leq j\leq n_i$. Let $\C[\eps_i]=\C[x]/(x^{n_i+1})$. Define $\bu_i=\sum u_i^{(j)}\eps_i^j$. We also have characters $\chi_0,\ldots,\chi_m$. Let $\chi_i(\sum\eps_i^j b_j)=i\Re\sum c_j^{(i)}b_j$. We assume that $\sum c_j^{(0)}=0$ for simplicity.

\begin{defn}
Let $\HH_x$ be the operator on $\cH$ defined by
    \begin{align*}
        \HH_x\psi(\bu_0,\ldots,\bu_m)=\int_{\C}\psi\bigg(\frac{t_0+\eps_0-x}{s-\bu_0},\ldots,\frac{t_m+\eps_m-x}{s-\bu_m}\bigg)\frac{\exp\bigg(\sum\chi_i\big(\log(s-\bu_i)\big)\bigg)ds\ovl{ds}}{\prod_{i=0}^m \abs{s-u_i^{(0)}}^{2n_i+2}}
    \end{align*}
    Let \[H_x = \prod_{i=0}^{m} |t_i-x|\exp\bigg(\tfrac12\chi_i\big(\log(t_i+\eps_i-x)\big)\bigg) \HH_x.\]
\end{defn}

 Note that \[\log(t_i+\eps_i-x)=\log(t_i-x)+\sum \eps_i^j \frac{(x-t_0)^{-j}}{j},\] hence \[\exp\bigg(\tfrac12\chi_i\big(\log(t_i+\eps_i-x)\big)\bigg)=\exp\big(i\Re c_0\log(t_i-x)+i\Re\sum \tfrac1j c^{(i)}_j (x-t_i)^{-j}\big).\] 

\subsection{The unitary representation}
\label{SecUnitaryRep}
Suppose that we have $k$ possibly glued points $t_0,\ldots,t_k$ and two unglued points $t_{k+1}, t_{k+2}$. We expect that the statements and proofs of this section  and the sections below can be modified to allow all non-infinite points to be glued. 

We will now show that $H_x = \int_{\CC} U_{s,x} d\nu(s)$, where $U_{s,x}$ are certain unitary operators on $\cH = L^2(\PP^{m-1}_{\CC})$ and $\nu$ is the same measure on $\C$ as in~\cite{EFK3,EFK4}.

 Set the unglued points $u_{k+1}=t_{k+1}=0$ and $u_{k+2}=t_{k+2} = 1$. A short computation gives that the resulting modified Hecke operator is
\begin{align*}
&(\HH_x \psi)(\bu_0,\dots,\bu_{k}) = \\
&\int_{\CC} \psi\paren{\frac{s(s-1)}{s-x}\paren{\frac{x}{s}+\frac{t_0+\eps_0-x}{s-\bu_0}, \cdots, \frac{x}{s}+\frac{t_k+\eps_k-x}{s-\bu_k}}} \\
&\ \ \ \ \ \ \ \cdot \frac{|s(s-1)|^{m-2} \exp\bigg(\sum\chi_i\big(\log(s-\bu_i)\big)\bigg)dsd\bar{s}}{|s-x|^m\prod_{l=0}^{k}|s-u_l^{(0)}|^{2n_l+2}},
\end{align*}
where $D_l$ is defined similarly to $D$.

The unitary operators $U_{s,x}$ will be given by the action of a group element 
\[g_{s, x} = (g_{s,x,0}, g_{s, x, n+1},\dots,g_{s, x, k}) \in \PGL_2(\CC[\eps_0]) \times \PGL_2(\CC[\eps_1])\times\cdots\times\PGL_2(\CC[\eps_k]).\]
The action of $\PGL_2(\CC[\eps_i])$ on the coordinates $u_i^{(0)},\dots,u_i^{(n_i)}$ is described in \cref{SecRepresentationsOfGEps}.

Then, using \cref{identities}, it is easy to check the following:

\begin{prop}
\label{unitary}
	We have $H_x = \int_{\CC} U_{s,x}d\nu(s)$, where $U_{s,x}$ is the unitary operator given by the action of the group element $g_{s,x} = (g_{s,x,0}, g_{s, x, 1},\dots,g_{s, x, k})$, where
	\[g_{s, x, i} = \begin{pmatrix}
    -(s-1)x & (t_i+\eps_i)s(s-1) \\
    -(s-x) & s(s-x)
\end{pmatrix},\] and $\nu(s) = |\frac{x(x-1)}{s(s-1)(s-x)}| dsd\bar{s}$. 
\qed
\end{prop}



\subsection{Boundedness}
\label{SecBounded}
Initially, the Hecke operators are only partially defined.
Let $V\subset \cH$ be the (dense) subset of \emph{continuous} functions $\psi$, translation-invariant and homogeneous of degree $-m$. Let $U\subset \CC^{m+1}$ be the subset of points where no two coordinates are equal to each other. 

\begin{prop}
\label{PropHxWellDefinedAndBounded}
	For $\psi\in V$, the integral $(\HH_x\psi)(u_0,\dots,u_m)$ converges and is continuous on $U$, and can be extended to an element of $\cH$.
\end{prop}

\begin{proof}
	We have to first show the integral converges, i.e.~to check the behavior of the formula in \cref{formula} at $s=u_0^{(0)}, u_{1}^{(0)},\dots,u_{k}^{(0)},u_{k+1},u_{k+2},\infty$. Let us use translation invariance to set the last coordinate $u_{k+2}=0$, and also without loss of generality set $t_{k+2}=0$. We obtain
	\begin{align*}
		\HH_x\psi(u_0^{(0)},\dots,u_{k+1}) &= \int_{\CC}\psi\paren{\frac{t_0s-xu_0^{(0)}}{s-u_0^{(0)}}, \partial_0\paren{\frac{t_0-x}{s-u_0^{(0)}}},\dots, \frac{1}{n!}\partial_0^{n_0}\paren{\frac{t_0-x}{s-u_0^{(0)}}}, \frac{t_{1}s-xu_{1}^{(0)}}{s-u_{1}^{(0)}}, \dots} \\
		&\ \ \ \ \ \ \ \ \ \ \cdot \frac{|s|^{m-2}\exp\bigg(\sum\chi_i\big(\log(s-\bu_i)\big)\bigg)dsd\bar{s}}{\prod_{l=0}^{k+1} |s-u_l^{(0)}|^{2n_l}},
	\end{align*}
 where $m+1$ is the total multiplicity of all non-infinite points.
    From this, it is clear that as $s\to\infty$, $\HH_x\psi(u_0,\dots,u_{m-1})$ decays as $|s|^{-m-2}$, hence integrable. To check the behavior as $s\to u_0^{(0)}$, we use homogeneity and scale all arguments up by\linebreak $(s-u_0^{(0)})^{n_0+1}$; then there will be an additional $|s-u_0^{(0)}|^{(n_0+1)m}$ term in the measure, so that as $s\to u_0^{(0)}$ the integral behaves as $|s-u_0^{(0)}|^{(n_0+1)m - (2n_0+2)}$ which is also integrable. A similar calculation addresses the behaviors at the other points.
	
	Continuity of $\HH_x\psi$ in $U$ follows from continuity of $\psi$. Finally, $\HH_x\psi$ is $L^2$-integrable by Cauchy-Schwarz and the fact that $\norm{H_x} \le \int_{\CC} 
	|\frac{x(x-1)}{s(s-1)(s-x)}| dsd\bar{s} < \infty$, which is a consequence of \cref{unitary}.
\end{proof}

\begin{prop}
	The Hecke operators $H_x$ extend to bounded, self-adjoint, mutually commuting operators on $\cH$, for $x\neq t_i,\infty$.
\end{prop}

\begin{proof}
	Boundedness follows from the previous proposition and $\norm{H_x} < \infty$.	

It is easy to check that $g_{s,x}^{-1} = g_{\sigma(s),x}$, where $\sigma(s) = \frac{x(s-1)}{s-x}$. This implies $U_{s,x}\dual = U_{\sigma(s),x}$. Also, the measure $d\nu(s)$ is invariant under the involution $s\mapsto \sigma(s)$. This implies that $H_x$ are self-adjoint.

Let $x_1,x_2$ be two distinct points distinct from $t_i,\infty$. The fact that operators $H_{x_1},H_{x_2}$ commute is a consequence of the general fact that Hecke modifications at distinct points $(x_1,s_1), (x_2,s_2)$ commute. Concretely, it can also be checked directly using \cref{unitary}; it can be reduced to the routine calculation that $d\nu_{x_1}(s_1)d\nu_{x_2}(s_2') = d\nu_{x_1}(s_1')d\nu_{x_2}(s_2)$, where 
\[d\nu_{x_i}(s) = \abs{\frac{x_i(x_i-1)}{s(s-1)(s-x_i)}} dsd\bar{s}\]
and $s_1' = \frac{s_2-1}{s_2-x_2}\cdot\frac{x_1s_2 - x_2s_1}{s_2-s_1}$ and symmetric for $s_2'$. Here, $s_1'$ is the coordinate of the parabolic line $s_1$ after Hecke modification at $(x_2,s_2)$, and vice versa.
\end{proof}

\subsection{Compactness} 
\label{SecCompactness}
\begin{prop}
\label{compact}
	The Hecke operators $H_x$ are compact and norm-continuous in $x$, for $x\neq t_i,\infty$.
\end{prop}

\begin{proof}
Using \cref{unitary}, the exact same argument as in (\cite{EFK3}, Proposition 3.13) goes through, the only thing that we have to show is that the rational map \[\phi_N: \AA^N_{\CC} \mapsto \mathbf{G}_{n,m} = \prod_{i=0}^k \PGL_2(\C[\eps_i])\] given by $(s_1,\dots,s_N)\mapsto g_{s_1,x}\cdots g_{s_N,x}$, where, say, $N=4m$, is dominant. This is shown in Lemma~\ref{LemDominant} below.
\end{proof}

\begin{rem}
    In particular, in Lemmas~\ref{LemElementsGeneratesl2} and~\ref{LemElementsGenerateProduct} below, we give a proof of group generation claimed in~\cite{EFK3} in the proof of Proposition~3.13 that used Lemma~8.9 therein. 
\end{rem}

We now prove \cref{LemDominant}. Denote $G=\PGL_2$.

\begin{lem}
\label{LemElementsGeneratesl2}
    For any $x\neq t\in \CC$, the elements 
	\[g(s) = g_{t,x}(s) = \begin{pmatrix}
		-(s-1)x & ts(s-1) \\
		-(s-x) & s(s-x)
	\end{pmatrix} \in G(\CC)\]
	generate a dense subgroup of $G(\CC)$, as $s$ ranges in $\CC\backslash\{0,1,x\}$.
\end{lem} 
\begin{proof}
    As $g(s)^{-1} = g(\frac{x(s-1)}{s-x})$, this set is closed under inverses and contains the identity. Let $\fg=\sl_2(\CC)$ be the Lie algebra of $G(\CC)$, and let $H$ be the closure of the subgroup that these elements generate. Then $H$ is a Lie group, so that we may consider its Lie algebra $\fh$. It suffices to show that $\fh=\fg$. By definition, $\fh$ contains the elements
\[g(s)^{-1}g'(s) = \frac{1}{s(s-1)(s-x)}
\begin{pmatrix}
	\frac{sx(t-1)}{t-x} - \frac{1}{2}(s^2+x) & \frac{s^2t(1-x)}{t-x} \\
	\frac{x(x-1)}{t-x} & \frac{1}{2}(s^2+x) - \frac{sx(t-1)}{t-x}
\end{pmatrix},\]
which linearly spans the 3-dimensional space $\fg=\sl_2(\CC)$. 
\end{proof}

\begin{lem}
\label{LemElementsGenerateGeps}
    Denote $\CC[\eps] = \CC[\eps]/(\eps^{n+1})$. The elements $g(s) = g_{t+\eps, x}(s)$ generate a dense subgroup of $G(\CC[\eps])$.
\end{lem} 
\begin{proof}
    Let $H$ be the closure of the subgroup they generate, and let $\fh$ be its Lie algebra, which lies in $\fg = \sl_2(\CC[\eps])$. It suffices to show $\fh=\fg$. We know $\fh$ contains the elements $A(s) = s(s-1)(s-x)g(s)^{-1}g'(s) = A_0(s) + A_1(s)\eps + \dots$, where
\[A_0(s) = \begin{pmatrix}
	\frac{sx(t-1)}{t-x} - \frac{1}{2}(s^2+x) & \frac{s^2t(1-x)}{t-x} \\
	\frac{x(x-1)}{t-x} & \frac{1}{2}(s^2+x) - \frac{sx(t-1)}{t-x}
\end{pmatrix}, \quad A_1(s) = \frac{x(1-x)}{(t-x)^2} \begin{pmatrix}
	s & -s^2 \\
	1 & -s
\end{pmatrix}.\]

Let us first produce an element $X\in\fh$ whose constant term is 0. Suppose we write $A_0(s) = \begin{pmatrix}
	a & b \\
	c & -a
\end{pmatrix}$, then 
\[A_1(s) = \begin{pmatrix}
	\frac{1-x}{(t-x)(t-1)}a + \frac{1}{2t(t-1)}b + \frac{1}{2(1-t)}c  & -\frac{x}{t(t-x)}b \\
	- \frac{1}{t-x}c & -(\frac{1-x}{(t-x)(t-1)}a + \frac{1}{2t(t-1)}b + \frac{1}{2(1-t)}c)
\end{pmatrix}.\]
It is not hard to verify that the commutator $[A(s_1+1)-A(s_1), A(s_2+1)-A(s_2)]$ is given by
\[\frac{4(s_1-s_2)t(t-1)x(x-1)}{(t-x)^2}\begin{pmatrix}
	0 & 1\\
	0 & 0
\end{pmatrix} + \frac{4(s_1-s_2) x(x-1)(2tx-t-x)}{(x-t)^3}\begin{pmatrix}
	0 & 1\\
	0 & 0
\end{pmatrix}\eps+\dots\]
which is linearly independent from the elements of the above form. Thus some linear combination of them would give $X\in \fh$ whose constant term is 0 and $\eps$ term is nonzero.

  
  Now that we have found one element $X\in \fh$, $X = B_1\eps+\dots$ with $B_1\neq 0$, consider its commutator with all the elements $A(s) = A_0+A_1\eps+\dots$. Since $[A(s), X] = [A_0,B_1]\eps+\dots$, and $\sl_2(\CC)$ is simple, by Step 1, we may now generate all elements of form $B_1\eps+\dots$, where $B_1\in \sl_2(\CC)$. Now, taking commutators once again, we can generate all elements of form $B_2\eps^2+\dots$, where $B_2\in \sl_2(\CC)$, and so on. Thus we have shown that $\fh = \sl_2(\CC[\eps]) = \fg$ as desired.
\end{proof}

\begin{lem}
\label{LemElementsGenerateProduct}
    The elements $g_{s,x} = (g_{t_0+\eps_0,x}(s), g_{t_{1}+\eps_1,x}(s),\dots,g_{t_{k}+\eps_k,x}(s))$, where $s\in \CC\backslash\{0,1,x\}$, generate a dense subgroup of $\mathbf{G}_{n,k} = \prod_{i=0}^k \PGL_2(\C[\eps_i])$.
\end{lem} 

\begin{proof}
    Use induction on $k$. The induction basis $k=0$ is already shown in \cref{LemElementsGenerateGeps}. For the induction step, let $H$ be the closure of the subgroup that $g_{s,x}$ generate. By induction hypothesis, $H$ surjects onto $\mathbf{G}_{n,k-1} = \prod_{i=0}^{k-1}\PGL_2(\C[\eps_i])$ (the first $k$ factors) and $G(\CC[\eps_k])$ (the last factor). By \cref{pavel}, it follows that $H\subset \mathbf{G}_{n,k-1}\times G(\CC[\eps_k])$ is the preimage of the graph of some smooth surjective map $f: \mathbf{G}_{n,k-1} \to G(\CC[\eps_k])/L$, where $L\lhd G(\CC[\eps_k])$ is the kernel of $H\to \mathbf{G}_{n,k-1}$. Since $G(\CC)$ is simple and any normal subgroup of $G(\CC[\eps_k])$ containing $G(\CC)$ coincides with $G(\CC[\eps_k])$, there are two options: either $L$ is inside the congruence subgroup $\begin{pmatrix}
        1+\eps_k\C[\eps_k] & \C[\eps_k]\\
        \C[\eps_k] & 1+\eps_k\C[\eps_k]
    \end{pmatrix}$ or $L=G(\CC[\eps_k])$. In the latter case $H = \mathbf{G}_{n,k}$ and we are done; in the former case, consider $pf$, the composition of $f$ with projection $G(\C[\eps_k])/L\to G(\C)$. This is a surjective map. Take any index $j$ such that $pf$ restricted to $G(\C[\eps_j])$ is nontrivial. Let $\phi$ be the restriction of $pf$ to $G(\C[\eps_j])$. Using \cref{C[eps]-maps} we see that $\phi$ is a composition of projection and an automorphism of $G(\C)$, they are all inner. Since the points $t_0,t_{1},\dots,t_{k}$ are all distinct, the functions of $s$ given by $\frac{\tr^2}{\det}$ of the matrix $g_{t_i,x}(s)$ are all distinct, which is a contradiction.
\end{proof}

\begin{lem}
\label{LemDominant}
    The rational map $\phi_N:\AA^N_{\CC}\to \mathbf{G}_{n,m}$ is dominant, where $N = 4m$.
\end{lem}

\begin{proof}
     Define $\phi_l$ similarly. Let $U_l$ be the Zariski-closure of the image of $\phi_l$, then it is a closed irreducible set in $\mathbf{G}_{n,k}$ of dimension at most $l$. Since $g_{s,x}g_{\sigma(s),x}= 1$, where $\sigma(s) = \frac{x(s-1)}{s-x}$, we have a chain $U_0\subset U_2\subset U_4\subset \dots$, and let $2l$ be the smallest index such that $U_{2l}=U_{2l+2}$. Then $U_{2l}=U_{2l+2}=\dots$, so $U_{2l}\supset H$, so by Step 3, $U_{2l} = \mathbf{G}_{n,k}$. Since $G_{n,k}$ has dimension $3(m-1)<4m$, $U_{4m} = \mathbf{G}_{n,m}$ as desired. 
\end{proof}

\subsection{Spectral decomposition}
\label{SecSpectralDecomp}
\begin{lem}
\label{LemAsymptoticsInfinity}
$H_x$ has asymptotics $2\abs{x}\log\abs{x}$ when $x$ tends to infinity. In other words, for any $\psi\in\cH$ we have $\lim_{x\to\infty}\frac{H_x\psi}{2\abs{x}\log\abs{x}}=\psi$.
\end{lem}
\begin{proof}
The proof is completely analogous to the proof of Proposition 3.15(i) in~\cite{EFK3}. Note that $\|x\|$ in~\cite{EFK3} is $\abs{x}^2$ in our case. We also note that due to asymmetry this proof works only for $x=\infty$ and other points with multiplicity one, where the limit is a certain involution $S_i$, but not $t_0$.
\end{proof}

By the spectral theorem for commuting compact self-adjoint operators, we conclude the following. 

\begin{cor}
\label{spectral}
	There is an orthogonal decomposition $\cH = \bigoplus_{l=0}^{\infty} \cH_l$, where $\cH_l$ are finite dimensional joint eigenspaces: for any $\psi_l\in \cH_l$, $H_x\psi_l = \beta_l(x)\psi_l$ where $\beta_l(x)$ are real-valued and continuous in $x$. 
\end{cor}

\begin{proof}
	Using Lemma~\ref{LemAsymptoticsInfinity} we get that the operators $H_x$ have trivial common kernel, so all $\cH_l$ are finite dimensional. Continuity of $\beta_l(x)$ follows from norm-continuity of $H_x$.
\end{proof}

\section{Gaudin Hamiltonians for $\PGL_2(\CC[\eps])$}
\label{SecGaudin}
In this section we will prove the analog of Corollary 4.14 in~\cite{EFK3}. Namely, we will prove that each eigenvalue $\beta_k$ satisfies an $\SL_2$-oper differential equation. 

We used an irreducible representation of the group $\PGL(2,\C[\eps])$ to obtain a compact formula for Hecke operators. Similarly, we will use Lie algebra $\sl_2(\C[\eps])$ to obtain a convenient expression for differential operators.

In this section we consider slightly more general situation than before and assume that we have $k$ distinct points in the limit and the infinite point. Suppose that $i$-th point is the result of gluing $n_i$ points and there was no gluing at infinity. Let $\C[\eps_i]=\C[x]/(x^{n_i})$.

We have the following compact expression for the generating function of classical Gaudin elements for $\sl_2$~\cite{FrenkelBetheAnsatz},~\cite{Sklyanin}: $\sum\frac{G_i}{x-t_i}=e(x)f(x)+\tfrac14h(x)^2+\tfrac12h'(x)$. Here for $a\in\sl_2$ we have \[a(x)=\sum\frac{a_i}{x-t_i},\] where $a_i$ is $a$ in the $i$-th tensor factor surrounded by ones. 

Consider the same expression $G(x)=e(x)f(x)+\tfrac14h(x)^2+\tfrac12h'(x)$ where for $a\in\sl_2$ we define \[a(x)=\sum_{i,j}\frac{\eps^j a_i}{(x-t_i)^{j+1}}.\] We can also think of this as a generating function of a commutative subalgebra of $\bigotimes U_{\C}(\sl_2(\C[\eps_i]))$ corresponding to a quantum integrable system:
\begin{lem}
\[G(x)=\sum_{i,j} \frac{G_i^{(j)}}{(x-t_i)^{j+1}},\] where $G_i^{(0)}=G_i$ are classical Gaudin elements and \[G_i^{(j)}=\sum_{a+b=j-1}\eps_i^a e\cdot \eps_i^b f+\eps_i^b f\cdot \eps_i^a e+\tfrac12\eps_i^a h\cdot \eps_i^b h.\]    
\end{lem}
Note that $G_i^{(j)}$ for $j\ge 0$ are in the center of $U_{\C}(\sl_2(\C[\eps]))$.

The formula from \cref{formula} can be written as
\[H_x\psi=\int_{\C} \rho(g_{s,x})\psi ds\ovl{ds},\] where $g_{s,x,i}=\begin{pmatrix}
0 & t_i-x+\eps_i\\
-1 & s
\end{pmatrix}$. The representation $\rho$ on $\bigotimes L^2(\Pone[\eps_i])$ is defined in the $i$-th coordinate as the irreducible representation of $\PGL_2(\C[\eps_i])$ in $L^2(\Pone[\eps_i])$ described in \cref{SecRepresentationsOfGEps}


The differential operators are obtained as an image of all $G_i^{(j)}$ in \[D((\Pone[\eps_1]\times\cdots\times\Pone[\eps_k])/B,\mc{L}),\] where $\mc{L}$ is a holomorphic line bundle described in \cref{SecRepresentationsOfGEps}. It can be checked that $G_i^{(j)}$ acts a number $l_i^{(j)}$ when $n_i\leq j<2n_i$ and acts as zero when $j\geq 2n_i$. Abusing notation, we will denote the differential operators also $G_i^{(j)}$.

By $D$ of the quotient we mean the quantum Hamiltonian reduction  of the algebra of differential operators. The only thing we need to check is that $G_i^{(j)}$ commute with $\sum e_l$, $\sum h_l$. For $j>0$ this is true because $G_i^{(j)}$ is a central element, for $j=0$ we can see it directly.

We are ready to prove the differential equation for Hecke operators:

\begin{prop}
\label{PropDifferentialEquation}
Let $\psi$ be a smooth function with compact support modulo translation and dilations. Then
\[\partial_x^2(H_x\psi)=G(x)\psi\] in the sense that both sides are defined and equal to each other for $x\neq t_i,\infty$ on the set of points where no two coordinates are equal to each other. 
\end{prop}
\begin{proof}
Denote $\rho(g_{s,x})$ by $g_{s,x}$ for convenience. We have $H_x\psi=\int_{\C} g_{s,x}\psi ds \ovl{ds}$. Using the standard lemma on differentiating under the integral sign and reasoning similarly to the proof of Proposition~\ref{PropHxWellDefinedAndBounded} we get \[\partial_x(H_x\psi)=\int_{\C}(\partial_x g_{s,x}g_{s,x}^{-1})g_{s,x}\psi ds\ovl{ds}.\] Here \[\partial_x g_{s,x,i}g_{s,x,i}^{-1}=\begin{pmatrix}
    0 & -1\\
    0 & 0
\end{pmatrix}\cdot\frac{1}{x-t_i-\eps_i}\begin{pmatrix}
    s & x-t_i-\eps\\
    1 & 0
\end{pmatrix}=\frac{1}{x-t_i-\eps_i}\begin{pmatrix}
    -1 & 0\\
    0 & 0
\end{pmatrix}\] is an element of $\sl_2(\C[\eps_i])$. We have $\begin{pmatrix}
    -1 & 0\\
    0 & 0
\end{pmatrix}=-\tfrac12h$, so that \[\partial_x g_{s,x,i}g_{s,x,i}^{-1}=-\frac{1}{2(x-t_i-\eps_i)}h=-\tfrac12\sum (x-t_i)^{-j-1}\eps_i^j h\] and \[\partial_x g_{s,x}g_{s,x}^{-1}=-\tfrac12 h(x).\]

Similarly computing the second derivative we get
\[\partial_x^2(H_x\psi)=(\tfrac14 h(x)^2-\tfrac12 h'(x))H_x\psi. \]
The differential operator $\tfrac14 h(x)^2-\tfrac12 h'(x)$ is not $B$-invariant, so it does not give a differential operator on $\prod \Pone[\eps_i]/B$. To fix this note that for any smooth function $\phi$ we have $\partial_s(g_{s,x}\phi)=(\partial_s g_{s,x})g_{s,x}^{-1}g_{s,x}\phi$. We have
\[\partial_s g_{s,x,i}g_{s,x,i}^{-1}=\begin{pmatrix}
    0 & 0\\
    0 & 1
\end{pmatrix}\cdot\frac{1}{x-t_i-\eps_i}\begin{pmatrix}
    s & x-t_i-\eps\\
    1 & 0
\end{pmatrix}=\frac{1}{x-t_i-\eps_i}\begin{pmatrix}
    0 & 0\\
    1 & 0
\end{pmatrix}.\] Similarly to the above we get $\partial_s g_{s,x}g_{s,x}^{-1}=f(x)$. Hence for any first-order differential operator $D$ we have \[(\tfrac14 h(x)^2-\tfrac12 h'(x))H_x\psi=(\tfrac 14 h(x)^2-\tfrac12h'(x)+f(x)D)H_x\psi.\] Taking $D=e(x)$ and using $[e(x),f(x)]=h'(x)$ we get the claim of the theorem.
\end{proof}
\begin{rem}
In particular, the proof of Proposition~\ref{PropDifferentialEquation} gives a simpler explicit proof of Proposition 4.3 in~\cite{EFK3} and Proposition 3.7 in~\cite{EFK4}.
\end{rem}

Arguing similarly to the proof of Proposition 4.6 of~\cite{EFK3} (with $\psi\in\cH_K$ instead of $\eta$) we get the following: for any $\psi\in\cH_k$ the distribution $G_i^{(j)}\psi$ equals to $\mu_{i,j,k}\psi$ for some complex number $\mu_{i,j,k}$ that depends on $\beta_k$ but not $\psi$.

Applying Proposition~\ref{PropDifferentialEquation} to $\psi\in\cH_k$ again we get the following:
\begin{cor}
The function $\beta_k(x)$ satisfies the differential equation
\[L(\mathbf{\mu}_k)\beta_k(x)=0,\]
where \[L(\mathbf{\mu}_k)=\partial_x^2+\sum_i\sum_{j=n_i}^{2n_i-1} \frac{l_i^{(j)}}{(x-t_i)^{j+1}}+\sum_{i,j<n_i}\frac{\mu_{i,j,k}}{(x-t_i)^{j+1}}\] is an $\SL_2$-oper.
\end{cor}

\subsection*{Acknowledgments} 

Most of this paper was written during the SPUR summer program at MIT. The authors thank Pavel~Etingof for suggesting this project and helpful remarks on the previous versions of this paper, David~Jerison for organizing the SPUR program, and Prof.~Etingof and Prof.~Jerison for many discussions and helpful suggestions. 

\bibliography{bib}

\end{document}

%% file: preamble.tex
\usepackage[utf8]{inputenc}
\usepackage{amsmath, amssymb, amsfonts, amsthm, stmaryrd, multicol, enumerate, graphicx, mdframed, setspace, fancyhdr, mathrsfs, hyperref, xr, todonotes, tikz, tikz-cd, xcolor}
\usepackage[capitalize]{cleveref}
\usepackage[english]{babel}
\usepackage[shortlabels]{enumitem}
\usepackage[margin=1in]{geometry}

\hypersetup{
    colorlinks=true,
    linkcolor=black,
    urlcolor=magenta
}
\makeatletter
\@addtoreset{equation}{section}
\renewcommand{\thesubsection}{%
  \ifnum\c@subsection<1 \@arabic\c@section
  \else \thesection.\@arabic\c@subsection
  \fi
}
\makeatother

\theoremstyle{definition}
    \newtheorem{defn}{Definition}[section]
    \newtheorem{exm}[defn]{Example}

\theoremstyle{plain}
    
    \newtheorem{lem}[defn]{Lemma}
    \newtheorem{prop}[defn]{Proposition}
    \newtheorem{cor}[defn]{Corollary}

\theoremstyle{remark}
    \newtheorem{rem}[defn]{Remark}

\renewcommand{\AA}{\mathbb{A}}
 
\newcommand{\CC}{\mathbb{C}} 
\newcommand{\C}{\mathbb{C}}

\newcommand{\HH}{\mathbb{H}}

\newcommand{\NN}{\mathbb{N}} 
 
\newcommand{\PP}{\mathbb{P}} 
 
\newcommand{\RR}{\mathbb{R}}

\newcommand{\bu}{\mathbf{u}}

\newcommand{\cE}{\mathcal{E}}

\newcommand{\cH}{\mathcal{H}}

\newcommand{\cK}{\mathcal{K}}

\newcommand{\cO}{\mathcal{O}}

\newcommand{\fg}{\mathfrak{g}}
\newcommand{\fh}{\mathfrak{h}}

\renewcommand{\sl}{\mathfrak{sl}}

\newcommand{\mc}{\mathcal}

\newcommand{\dual}{^{\ast}}

\renewcommand{\bar}{\overline}
\newcommand{\ovl}{\overline}

\newcommand{\into}{\hookrightarrow}

\newcommand{\eps}{\varepsilon}

\newcommand{\abs}[1]{\left | #1 \right |}  
\newcommand{\norm}[1]{\left\| #1 \right\|} 
 
\newcommand{\paren}[1]{\left( #1 \right)}
\newcommand{\Pone}{\mathbb{P}^1}

\DeclareMathOperator{\Aut}{Aut} 

\DeclareMathOperator{\Bun}{Bun}

\renewcommand{\Im}{\mathrm{Im\thinspace}}

\DeclareMathOperator{\PGL}{PGL}

\renewcommand{\Re}{\mathrm{Re\thinspace}}

\DeclareMathOperator{\SL}{SL}

\DeclareMathOperator{\tr}{tr}

%% file: WildRamificationArxiv.bbl
\begin{thebibliography}{10}

\bibitem{BK}
Alexander Braverman and David Kazhdan.
\newblock Some examples of {H}ecke algebras for two-dimensional local fields.
\newblock {\em Nagoya Mathematical Journal}, 183:57–84, 2006.

\bibitem{EFK1}
Pavel Etingof, Edward Frenkel, and David Kazhdan.
\newblock An analytic version of the {L}anglands correspondence for complex
  curves.
\newblock In {\em Integrability, quantization, and geometry {II}. {Q}uantum
  theories and algebraic geometry}, volume 103 of {\em Proc. Sympos. Pure
  Math.}, pages 137--202. Amer. Math. Soc., Providence, RI, 2021.

\bibitem{EFK3}
Pavel Etingof, Edward Frenkel, and David Kazhdan.
\newblock Analytic {L}anglands correspondence for {$\PGL_2$} on
  {$\mathbb{P}^1$} with parabolic structures over local fields.
\newblock {\em Geom. Funct. Anal.}, 32(4):725--831, 2022.

\bibitem{EFK4}
Pavel Etingof, Edward Frenkel, and David Kazhdan.
\newblock A general framework for the analytic {L}anglands correspondence,
  2023, 2311.03743.

\bibitem{EFK2}
Pavel Etingof, Edward Frenkel, and David Kazhdan.
\newblock {Hecke operators and analytic Langlands correspondence for curves
  over local fields}.
\newblock {\em Duke Mathematical Journal}, 172(11):2015 -- 2071, 2023.

\bibitem{FrenkelBetheAnsatz}
Edward Frenkel.
\newblock {Affine algebras, Langlands duality and Bethe ansatz}.
\newblock In {\em {11th International Conference on Mathematical Physics
  (ICMP-11) (Satellite colloquia: New Problems in the General Theory of Fields
  and Particles, Paris, France, 25-28 Jul 1994)}}, 6 1995, q-alg/9506003.

\bibitem{Ko}
Maxim Kontsevich.
\newblock Notes on motives in finite characteristic.
\newblock {\em Progress in Mathematics}, 270, 03 2007.

\bibitem{La}
R.P. Langlands.
\newblock On analytic form of geometric theory of automorphic forms (in
  {R}ussian).

\bibitem{Sklyanin}
E.~K. Sklyanin.
\newblock {Separation of variables in the Gaudin model}.
\newblock {\em Zap. Nauchn. Semin.}, 164:151--169, 1987.

\bibitem{Te}
Jörg Teschner.
\newblock {Quantization of the Quantum Hitchin System and the Real Geometric
  Langlands Correspondence}.
\newblock In {\em {Geometry and Physics: Volume I: A Festschrift in honour of
  Nigel Hitchin}}. Oxford University Press, 10 2018,
  https://academic.oup.com/book/0/chapter/367234564/chapter-pdf/45150834/oso-9780198802013-chapter-13.pdf.

\end{thebibliography}
